\documentclass[11pt]{article}

\usepackage{latexsym}
\usepackage{amssymb}
\usepackage{graphicx}
\usepackage{color}
\usepackage{amsmath}
\usepackage{float}
\usepackage{mathrsfs} %fonts veramente calligrafici con \mathscr{B}

\newtheorem{theorem}{Theorem}

\newtheorem{corollary}[theorem]{Corollary}
\newtheorem{definition}[theorem]{Definition}

\newtheorem{lemma}[theorem]{Lemma}

\newtheorem{proposition}[theorem]{Proposition}
\newtheorem{remark}[theorem]{Remark}

\newenvironment{proof}[1][Proof]{\textbf{#1.} }{\ \rule{0.5em}{0.5em}}
\floatstyle{ruled} \restylefloat{figure} \restylefloat{table}

\newcommand{\kom}[1]{}
\renewcommand{\kom}[1]{{\bf [#1]}}
\definecolor{blau}{rgb}{0.1,0.0,0.9}

\newcounter{komcounter}
\numberwithin{komcounter}{section}

\def\C{\mathcal C}

\def\e{\epsilon}

\def\I{\mathcal I}

\def\M{\mathcal M}

\def\O{\mathcal O}

\def\P{\mathcal P}

\def\R{\mathbb R}

\def\S{\mathcal S}

\def\0{\boldsymbol 0}

\def \subjet{\P^{2,-}_{ \Omega}}

\def \supjetbar{\bar \P^{2,+}_{\Omega}}
\def \subjetbar{\bar \P^{2,-}_{\Omega}}

\def \supjetbarbar {\bar \P^{2,+}_{\bar  \Omega}}

\begin{document}

\title{Systems of fully nonlinear parabolic obstacle problems \\with Neumann boundary conditions}

\author{Niklas L.P. Lundstr\"om, Marcus Olofsson\thanks{%
Department of Mathematics and Mathematical Statistics, Ume\aa\ University,
SE-901 87 Ume\aa , Sweden. E-mail: niklas.lundstrom@umu.se, marcus.olofsson@umu.se}}
\maketitle

\begin{abstract}
\noindent
We prove the existence of a unique viscosity solution to certain systems of fully nonlinear parabolic partial differential equations with interconnected obstacles in the setting of Neumann boundary conditions. The method of proof builds on the classical viscosity solution technique adapted to the setting of interconnected obstacles and construction of explicit viscosity sub- and supersolutions as bounds for Perron's method. Our motivation stems from so called optimal switching problems on bounded domains.
\medskip


\end{abstract}
%
%  35Q93: PDEs in connection with control and optimization
%
%  35D40: "Viscosity solutions to PDEs" (MSC2020)
%
%  35B50: "Maximum principles in context of PDEs" (MSC2020)
%
%  49L25  "Existence" "Viscosity solutions"
%
%%%%%%%%%%%%%%%%%%%%%%%%%%%%%%%%%%%%%%%%%%%%%
%%%%%%%%%%%%%%%%%%%%%%%%%%%%%%%%%%%%%%%%%%%%%
%%%%%%%%%%%%%%%%%%%%%%%%%%%%%%%%%%%%%%%%%%%%%
%%%%%%%%%%%%%%%%%%%%%%%%%%%%%%%%%%%%%%%%%%%%%
%%%%%%%%%%%%%%%%%%%%%%%%%%%%%%%%%%%%%%%%%%%%%
%%%%%%%%%%%%%%%%%%%%%%%%%%%%%%%%%%%%%%%%%%%%%

\setcounter{equation}{0} \setcounter{theorem}{0}

\section{Introduction}

This note deals with existence and uniqueness of viscosity solutions of systems of non-linear partial differential equations (PDE) with interconnected obstacles and Neumann boundary conditions.
Let $\Omega \subset\mathbb{R}^{n}$ be a domain, i.e. an open connected set,
and let $i\in \{1,2,...,m\}$ for some positive integer $m$.
The problem considered can in general be stated as
\begin{align} \label{problem}	
\min  \left\{ \partial _{t}u_{i}\left(t,x\right) + {F}_i\left(t,x,u_i(t,x), Du_{i}, D^{2}u_{i} \right),  u_{i}(t,x) - \M_i u(t,x) \right\} &= 0 && \text{in} \quad (0,T)\times\Omega,\notag\\
%\TCItag{$\ast$}
u_{i}(0,x) &= g_{i}\left( x\right) && \text{on} \quad \bar \Omega, \tag{BVP} \\
B_i\left(t,x,u_i(t,x),Du_{i} \right) &= 0  &&\text{on} \quad (0,T)\times\partial\Omega.\notag
\end{align}
%
%for $t\in [0,T]$, $x\in \overline \Omega \subset\mathbb{R}^{n}$, $i\in \{1,2,...,m\}$ and
The solution is a vector-valued function $u:=(u_1,u_2, \dots,u_m)$
where the components are interconnected through the obstacle $\M_\cdot$ in the sense that component $i$ lies above the obstacle
$$
%\M_iu(t,x) = \max_{i\neq j } \{u_i(t,x) -( u_j(t,x) -c_{ij}(t,x))\},
\M_iu(t,x) := \max_{i\neq j }\{u_j(t,x) -c_{ij}(t,x)\},
$$
which itself depends on all other components. Here and in the following, $Du_i \in \R^n$ and $D^2u_i \in \S^n$ ($\S^n$ being the set of symmetric $n \times n$ matrices) denote the gradient and Hessian matrix of $u_i$ w.r.t. $x$.

In the standard setting of a single PDE without obstacle,
the well-renowned paper Crandall-Ishii-Lions \cite{CIL92} proves existence and uniqueness of viscosity
solutions of fully non-linear elliptic PDEs with Neumann boundary conditions in smooth domains satisfying an exterior ball condition.
Dupuis-Ishii \cite{DI90, DI91} consider more general domains such as non-smooth boundaries and domains with corners, and
Barles \cite{B93} proved a comparison principle and existence of unique solutions to degenerate elliptic
and parabolic boundary value problems with non-linear Neumann type boundary conditions
in bounded domains with $W^{3,\infty}$-boundary.
Ishii and Sato \cite{IS04} proved similar theorems for
boundary value problems for some singular degenerate parabolic partial differential equations
with non-linear oblique derivative boundary conditions in bounded $C^1$-domains.
Further, in bounded domains with $W^{3,\infty}$-boundary, Bourgoing [4] considered singular degenerate parabolic
equations and equations having $L^1$-dependence in time.
Lundstr\"om-\"Onskog \cite{LO19} recently extended parts of \cite{DI90} by establishing existence and uniqueness of viscosity solutions to a general parabolic PDE in non-smooth, time-dependent domains.

For systems of variational inequalities with interconnected obstacles as in \eqref{problem},
the literature on bounded spatial domains $\Omega$ with boundary conditions is, to the authors' knowledge, scarce.
Instead much focus has been directed to unbounded domains (essentially $\Omega \equiv \mathbb R^n$) and often linear PDEs.
In this setting the literature is, on the contrary, rather rich, see, e.g.,
El--Asri-Hamadene \cite{EH09},
Djehiche-Hamadene-Popier \cite{DHP10},
Hu-Tang \cite{HT10},
Biswas-Jakobsen-Karlsen \cite{BJK10},
Hamadene-Morlais \cite{HM13},
Lundstr\"om-Nystr\"om-Olofsson \cite{LNO14a, LNO14b},
Lundstr\"om-Olofsson-\"Onskog \cite{LOO19},
Fuhrman-Morlais \cite{FM20},
Reisinger-Zhang \cite{RZ20}
and the many references listed therein. Extensions to more general operators have been studied recently in Klimsiak  \cite{K19}. Moreover, it has come to our attention that Boufoussi-Hamadene-Jakani \cite{BHJ22} recently considered a similar setting and prove existence and uniqueness of a viscosity solution to a system of PDEs with interconnected obstacles and Neumann boundary conditions. In contrast to the present paper, their method is based on a probabilistic formulation in terms of reflected backward stochastic differential equations.

A reason for this large amount of interest is the close connection to stochastic optimization and so called ``optimal switching'' problems. Under assumptions, the solution to \eqref{problem} then represents the value function of an optimal switching problem in which component $i$ of the underlying stochastic process has the second order operator $F_i$ as infinitesimal generator. With this interpretation, the domain $\Omega$ represents the constraints put on the underlying stochastic process and, in particular, $\Omega \equiv \mathbb R^n$ means no constraints.

However, when studying applications of optimal switching theory, one can very easily think of examples where constraints are not only natural, but also necessary for the model to be meaningful. For example, consider a hydro power plant with a dam in which the underlying (stochastic) process $X_t$ represents the amount of water currently available in the dam. Then, naturally, the process can only take values between $0$ and some $X_{max}$, the latter being the capacity of the dam. (In case the power plant is of run-of-river type, i.e., has no dam, the problem can be reduced to $\Omega =\mathbb R^n$ and an application of the theory in this setting was recently studied in Lundstr\"om-Olofsson-\"Onskog \cite{LOO20}.)
When the dam is full, additional increase of water must be spilled, keeping the amount of water in the dam constant. Such situation may be modelled through a reflection of the underlying stochastic process $X_t$. If this reflection is assumed in the normal direction, then one obtains a Neumann boundary condition, i.e.
\begin{align}\label{eq:BC}
B_i(t,x,u_i(t,x),Du_i) = \langle n(x), Du_i\rangle - f_i(x) = 0
\tag{BC}
\end{align}
where $n(x)$ is the normal of $\partial\Omega$ at $x$.
A more general reflection model puts us in the oblique derivative problem in which $n(x)$ in \eqref{eq:BC} should be replaced by
$\nu(x)$, a vector field satisfying $\langle \nu(x), n(x)\rangle > 0$ on $\partial\Omega$. The Dirichlet setting, which is studied in Barkhudaryan-Gomes-Shahgholian-Salehi \cite{BGSS20}, can be given a similar interpretation as above but then in the sense that the game ends when the process hits the boundary $\partial \Omega$.

Optimal switching problems with reflection is not as well-studied as the non-reflected counterpart, and the link to systems like \eqref{problem} is in general not yet rigorously established. However, given the close connection to real-world applications and the growing interest of optimal switching problems under constraints, see, e.g., Kharroubi \cite{K16}, we expect this to change and a thorough analysis of the related PDE-theory on bounded domains is motivated. Moreover, although the study is motivated by applications, the results are of independent mathematical interest, enlarging the class of problem for which the viscosity theory is fully investigated.

\setcounter{equation}{0} \setcounter{theorem}{0}

\section{Mathematical problem formulation, assumptions and main results}

\subsection{Notation and assumptions}

We let $\Omega \subset \R^n$ be an open connected bounded set with closure $\bar \Omega$ and boundary $\partial \Omega := \bar \Omega \setminus \Omega$ and consider $t \in [0,T]$.
We denote by $\mathcal C^{\alpha, \beta}$ the class of functions which are $\alpha$ and $\beta$ times continuously differentiable on
$[0,T] \times \bar \Omega$
w.r.t. the first (time $t$) and second (space $x$) variable, respectively.
If $f(\cdot)$ only depends on the spatial variable we simply write $f\in \mathcal C^\beta$. We assume that
\begin{equation} \label{ass:D1}
\mbox{$\partial \Omega \in \mathcal C^{1}$ and satisfies the exterior ball condition}, \tag{D1}	
\end{equation}
i.e.,
that $\partial \Omega$ is once continuously differentiable and that
$$
\mbox{for every $\hat x \in \partial \Omega$ there exist $x$ and $r>0$ s.t. $B(x,r) \not \subset \Omega$ and $\hat x \in \partial B(x,r)$},
$$
where $B(x,r):= \{y: |y-x|<r\}$. (A simple example of such a domain is $\partial \Omega \in \mathcal C^{1}$ with Lipschitz continuous derivative,
as such domains satisfy both the exterior and interior ball condition,
see Aikawa-Kilpel\"ainen-Shanmugalingam-Zhong \cite[Lemma 2.2]{AKSZ07}.)

We assume that $F_i(t,x,r,p,X) : [0,T] \times \bar\Omega \times \R \times \R^n \times \S^n \to \R$ is a continuous function satisfying
\begin{align} \label{ass:F1}
&\quad r \mapsto F_i(t,x,r,p,X)- \lambda \, r \, \mbox{is non-decreasing for some $\lambda \in \R$}, \tag{F1}\\
\label{ass:F2}
& \quad \mbox{$|F_i(t,x,r,p,X)-F_i(t,x,r,q,Y)| \leq \omega(|p-q| + ||X-Y||)$ for some $\omega:[0,\infty) \to [0,\infty)$} \notag \\
&\quad \mbox{with $\omega(0+)=0$} \tag{F2}
\end{align}
and that
\begin{equation} \label{ass:F3}
F_i(t,y,r,p,Y)-F_i(t,x,r,p,X) \leq \omega\left(\frac{1}{\e} |x-y|^2 + |x-y|(|p|+1)\right) \tag{F3}
\end{equation}
for every $t \in [0,T)$ fixed %(with the same modulus $\omega$)
and whenever
$$
\left(
\begin{matrix}
   X    &0 \\
    0       & -Y
\end{matrix} \right) \leq
\frac{3}{\e} \left (
\begin{matrix}
   I    &-I \\
    -I       & I
\end{matrix} \right),
$$
where $I\equiv I^n$ denotes the $n \times n$ identity matrix.
Note that these assumptions on $F_i$ imply degenerate ellipticity.

Regarding the boundary condition $B_i$, we assume that
\begin{align}\label{ass:BC1}
& B_i(t,x,r,p):=\langle n(x),p\rangle + f_i(t,x,r)\notag \\
&\mbox{where $n(x)$ is the exterior unit normal at $x$ and $f_i$ is continuous on $[0,T]\times \partial \Omega \times \mathbb{R}$,} \notag \\
&\mbox{and non-decreasing in $r$  for every $(t,x) \in [0,T] \times \partial \Omega$.}\tag{BC1}
\end{align}

\noindent
Last, we assume that the obstacle functions $c_{ij}(t,x)$ are continuous on $[0,T]\times \bar \Omega$ and satisfy
\begin{align}\label{ass:O1}
 (i)&\quad \mbox{$c_{ij}(t,x) \in \mathcal C^{1,2} ( [0,T]\times \bar \Omega $}), \notag \\
 (ii)&\quad c_{ii}(t,x)=0\mbox{ for each $i\in\{1,\dots,m\}$}, \tag{O1}	
\end{align}
and that the so called ``no-loop''-condition holds;
\begin{align} \label{ass:O2}
&\quad \mbox{for any sequence of indices $i_1$,\dots, $i_k$, $i_j\in\{1,\dots,m\}$ for each $j\in\{1,\dots,k\}$,}\notag\\
 &\quad \mbox{we have $c_{i_1 i_2}(t,x)+c_{i_2 i_3}(t,x)+\dots+c_{i_{k-1} i_k}(t,x)+c_{i_k i_1}(t,x)>0$},\notag\\
 &\quad \mbox{for all $(t,x)\in [0,T] \times \bar \Omega$} \tag{O2}.
\end{align}
As with \eqref{ass:D1}, we need to strengthen the assumption on the obstacle slightly to prove existence. In particular, we then also demand that
\begin{equation}\label{ass:O3}
c_{ij}(t,x) + c_{jk}(t,x) \geq c_{ik}(t,x) \quad \mbox{for all $i,j,k \in \{1,2, \dots, n\}$} \tag{O3}
\end{equation}
which is an additional structural condition compared to $\eqref{ass:O2}$.

Last, we require that the initial data is continuous and compatible with the obstacle functions, in particular that
\begin{equation} \label{ass:compatibility}
\mbox{$g_i(x) \geq g_j(x)- c_{ij}(0,x)\;\;$ for any $i,j \in \{1,\dots,m\}$.} \tag{I1}
\end{equation}

\begin{remark} \label{remark:F1}
In the parabolic setting studied here, it is a standard task to show that one can assume $\lambda >0$ in \eqref{ass:F1} w.l.o.g.
Indeed, if $\lambda \leq 0$, then for $\bar \lambda  < \lambda$, $u(t,x)$ is a solution to \eqref{problem} if and only if $e^{\bar \lambda  t}u(t,x)$ is a solution to \eqref{problem} with $F_i$ and $f_i$ replaced by
\begin{equation*}
-\bar{\lambda}r+e^{\bar{\lambda}t}F_i(t,x,e^{-\bar{\lambda}t}r,e^{-\bar{\lambda%
}t}p,e^{-\bar{\lambda}t}X)\quad \text{and}\quad e^{\bar{\lambda}t}f_i(t,x,e^{-%
\bar{\lambda}t}r),
\end{equation*}
and with $c_{ij}$ similarly scaled. The function to the left in the above display is strictly increasing and hence satisfies $\eqref{ass:F1}$ with $\lambda$ positive. A consequence of this is that we may as well assume that $F_i$ is strictly increasing in $r$ and that there exists $\gamma >0$ s.t. for $r>s$
 \begin{equation} \label{eq:F1star}
 \gamma (r-s) \leq F_i(t,x,r,p,X)  - F_i(t,x,s,p,X). \tag{F1*}
 \end{equation}
\end{remark}

We are looking for solutions in the viscosity sense and will use the classical definitions of (parabolic) sub- and superjets, provided here for convenience.
\begin{definition}[Definition (8.1) of  \cite{CIL92}]
For  $(\hat t, \hat x) \in (0,T) \times \mathcal O  $, the triplet $(a,p,X) \in \R \times \R^n \times \S^n$ lies in the \textit{parabolic superjet} of $u$ at $(\hat t, \hat x)$, written $(a,p,X) \in \mathcal P^{2,+}_{\mathcal O} u(\hat t, \hat x)$, if
$$u(t,x) \leq u (\hat t, \hat x) + a (t-\hat t) + \langle p, x-\hat x\rangle + \frac{1}{2} \langle X(x-\hat x), x- \hat x\rangle + o(|t-\hat t| + |x- \hat x|^2)$$
as $(0,T) \times \mathcal O \ni (t,x) \to (\hat t, \hat x)$.
Analogously, the parabolic subjet is defined by $\mathcal P^{2,-}_{\mathcal O} u(\hat t, \hat x) := -\mathcal P^{2,+}_{\mathcal O} -u(\hat t, \hat x)$.

The closure of $\mathcal P^{2,\cdot}_{\mathcal O} u (\hat t, \hat x)$, denoted $\bar{\mathcal P}^{2,\cdot}_{\mathcal O} u(\hat t, \hat x)$, is defined as
\begin{align*}
\bar{\mathcal P}^{2,\cdot}_{\mathcal O} u(\hat t, \hat x) :=& \{ (a,p,X) \in \R \times \R^n \times \S^n : \exists (t_n,x_n,p_n,X_n)\in (0,T) \times \mathcal O \times  \R^n \times \S^n \\
&\mbox{s.t. } (a_n, p_n, X_n) \in \mathcal P^{2,\cdot}_{\mathcal O}u(t_n,x_n) \mbox{ and } (t_n,x_n, u(t_n,x_n), a_n, p_n, X_n) \to (\hat t, \hat x, u(\hat t, \hat x), a,p,X) \}.
\end{align*}
\end{definition}

\begin{definition} \label{def:solution}
A function $u=(u_1, \dots, u_m) \in USC([0,T]\times\bar \Omega)$ is a \textit{viscosity subsolution} to \eqref{problem} if, for all $i \in \{1,\dots,m\}$,
\begin{align*}
(i)&\quad \mbox{$\min \left\{a+ {F_i}(t,x,u(t,x), p,X),  u_{i}(x,t) - \M_i u(t,x) \right\} \leq 0,$}  \quad \mbox{$ (t,x) \in (0,T) \times \Omega, (a,p,X) \in \supjetbar u_i(t,x)$} \\
(ii)&\quad \mbox{$\min \left\{ a + {F_i}(t,x,u(t,x), p, X ),  u_{i}(t,x) - \M_i u(t,x) \right\}$}\wedge \mbox{$ \langle n(x), p \rangle + f_i(t,x,u(x)) \leq 0 $}, \\
&\quad  \mbox{$(t,x) \in (0,T) \times \partial \Omega, (a,p,X) \in \supjetbarbar u_i(t,x)$ } \\
(iii)& \quad \mbox{$u_i(0,x) \leq g_i(x)$, $x\in \bar \Omega$}
\end{align*}

Viscosity \textit{supersolutions} are defined analogously with $USC$ replaced by $LSC$, ${\bar \P^{2,+}_{\cdot}}$ replaced by ${\bar \P^{2,-}_{\cdot}}$, $\vee$ replaced by $\wedge$, and $\leq$ replaced by $\geq$. A function $u \in C([0,T]\times\bar \Omega)$ is a \textit{viscosity solution} to \eqref{problem} if it is both a sub- and a supersolution.
\end{definition}

\subsection{Main results}

With the above preliminaries set, we are ready to state our main results, which concern comparison of sub- and supersolutions and existence and uniqueness of viscosity solutions to \eqref{problem}.

\begin{theorem} \label{thm:comparison}
Assume that \eqref{ass:D1}, \eqref{ass:F1} - \eqref{ass:F3}, \eqref{ass:BC1} and \eqref{ass:O1} - \eqref{ass:O2} hold. If $u$ and $v$ are viscosity sub- and supersolutions to \eqref{problem}, respectively, then $u(t,x) \leq v(t,x)$ for $(t,x) \in [0,T) \times \bar \Omega$.
\end{theorem}

\begin{theorem}\label{thm:existence}
Assume that \eqref{ass:D1}, \eqref{ass:F1}- \eqref{ass:F3}, \eqref{ass:BC1} and \eqref{ass:O1} - \eqref{ass:O3} and \eqref{ass:compatibility} hold. Then, there exists a unique viscosity solution to \eqref{problem}.
\end{theorem}

Our proofs rely on by now classical techniques in the theory of viscosity solutions, using doubling of variables and the maximum principle for semi-continuous functions for Theorem \ref{thm:comparison} and Perron's method together with a construction of certain sub- and supersolutions for Theorem \ref{thm:existence}. From the proof of Theorem \ref{thm:comparison} we also get the following corollary, useful in its own merit.
\begin{corollary}  \label{corr:mixedcomparison}
Let $u$ and $v$ be viscosity sub- and supersolutions, respectively, to \eqref{problem}. If $u$ and $v$ satisfy
Definition \ref{def:solution} $(ii)$ on some open region $G \subset (0,T) \times \partial \Omega$ and $u \leq v$ on $\left((0,T) \times \partial \Omega \right)\setminus G$, then $u(t,x) \leq v(t,x)$ for all $(t,x) \in [0,T) \times \bar \Omega$.
\end{corollary}

\begin{remark}
We have chosen to present our results and proofs in a rather simplistic setting, with directional derivative in the normal direction and a smooth, stationary domain $\Omega$. However, our proofs reveal that the additional arguments needed to treat systems of PDEs with interconnected obstacles (rather than a single PDE) are more or less decoupled from the standard arguments of Crandall-Ishii-Lions \cite{CIL92}.
Hence, by combining these additional arguments with existence and uniqueness proofs for more general PDEs and domains, one should be able to prove existence and uniqueness for systems of PDEs with interconnected obstacles in similar generality.
For tractability, we refrain from such generalizations here.
\end{remark}

%%%%%%%%%%%%%%%%%%%%%%%%%%%%%%%%%%%%%%%%%%%%%
%%%%%%%%%%%%%%%%%%%%%%%%%%%%%%%%%%%%%%%%%%%%%
%%%%%%%%%%%%%%%%%%%%%%%%%%%%%%%%%%%%%%%%%%%%%
%%%%%%%%%%%%%%%%%%%%%%%%%%%%%%%%%%%%%%%%%%%%%
%%%%%%%%%%%%%%%%%%%%%%%%%%%%%%%%%%%%%%%%%%%%%
%%%%%%%%%%%%%%%%%%%%%%%%%%%%%%%%%%%%%%%%%%%%%
%%%%%%%%%%%%%%%%%%%%%%%%%%%%%%%%%%%%%%%%%%%%%
%%%%%%%%%%%%%%%%%%%%%%%%%%%%%%%%%%%%%%%%%%%%%
%%%%%%%%%%%%%%%%%%%%%%%%%%%%%%%%%%%%%%%%%%%%%
%%%%%%%%%%%%%%%%%%%%%%%%%%%%%%%%%%%%%%%%%%%%%
%%%%%%%%%%%%%%%%%%%%%%%%%%%%%%%%%%%%%%%%%%%%%
%%%%%%%%%%%%%%%%%%%%%%%%%%%%%%%%%%%%%%%%%%%%%
%%%%%%%%%%%%%%%%%%%%%%%%%%%%%%%%%%%%%%%%%%%%%
%%%%%%%%%%%%%%%%%%%%%%%%%%%%%%%%%%%%%%%%%%%%%
%%%%%%%%%%%%%%%%%%%%%%%%%%%%%%%%%%%%%%%%%%%%%
%%%%%%%%%%%%%%%%%%%%%%%%%%%%%%%%%%%%%%%%%%%%%
%%%%%%%%%%%%%%%%%%%%%%%%%%%%%%%%%%%%%%%%%%%%%
%%%%%%%%%%%%%%%%%%%%%%%%%%%%%%%%%%%%%%%%%%%%%

\setcounter{equation}{0} \setcounter{theorem}{0}

\section{Proof of Theorem \ref{thm:comparison}: The comparison principle.}

Throughout the section, we assume that the assumptions of Theorem \ref{thm:comparison} hold, i.e., that \eqref{ass:D1}, \eqref{ass:F1}- \eqref{ass:F3}, \eqref{ass:BC1} and \eqref{ass:O1} - \eqref{ass:O3} hold.
Let $u=(u_1, u_2, \dots, u_m)$ and $v=(v_1, v_2, \dots, v_m)$ be sub- and supersolutions to \eqref{problem}, respectively.
Our proofs follow the classical outline of Crandall-Ishii-Lions \cite{CIL92} and consists of four major steps.

\subsubsection*{Step 1. An $\epsilon$-room on the boundary.}
\begin{lemma}[Lemma 7.6 of \cite{CIL92}] \label{lemma:varphi}
For any continuous $\nu(x): \partial \Omega \to \R^n$ satisfying $\langle n(x), \nu(x)	\rangle > 0$ there exists a $C^2(\bar \Omega)$ function $\varphi(x)$ s.t.
$$
\langle \nu(x), D\varphi(x) \rangle \geq 1 \mbox{ for } x \in \partial \Omega \quad 	\mbox{and} \quad \varphi(x) \geq 0 \mbox{ on } \bar \Omega.
$$
\end{lemma}

Let
$$
u_i^\eta(t,x) := u_i(t,x) - \eta \varphi(x) - C_\eta \quad \mbox{and} \quad v_i^\eta:= v_i(t,x) + \eta \varphi(x) + C_\eta
$$
where $\varphi(x)$ is as given by the lemma above (with $\nu(x) = n(x)$) and $C_\eta>0$ is a constant to be specified later.
Note that $u_i^\eta < u_i$ so that for any $(a,p,X) \in \supjetbar u_i^\eta (t,x)$ we have from \eqref{ass:F1} and Remark \ref{remark:F1} (i.e., from \eqref{eq:F1star}) that
\begin{align*}
&\gamma (\eta \varphi(x) + C_\eta) \leq F_i(t,x,u_i(t,x), p, X) - F_i(t,x,u_i^\eta(t,x),p, X)
\end{align*}
for some $\gamma >0$ and thus
\begin{align*}
a+ F_i(t,x,u^\eta(t,x),p, X) \leq&a+   F_i(t,x,u_i(t,x), p, X)  - \gamma \eta \varphi(x) - \gamma C_\eta \\
\leq & a+ F_i(t,x, u_i(t,x), p + \eta D\varphi(x), X + \eta D^2\varphi(x))-  \gamma \eta \varphi(x) - \gamma C_\eta  + \omega(\eta M)
\end{align*}
where $M= \sup_{\bar \Omega}\{D\varphi(x) +||D^2\varphi(x)||\}$ and where the last inequality comes from \eqref{ass:F2}.
Since $(a,p,X) \in \supjetbar u^\eta (t,x)$ and $u^\eta = u - \eta \varphi(x)- \C_\eta$ we have
$$
(a,p+\eta D\varphi(x), X + \eta D^2\varphi(x)) \in \supjetbar u(t,x)
$$
and therefore choosing $C_\eta = \frac{\omega(\eta M)}{\gamma}$ we get that
$$
a+ F_i(t,x,u_i^\eta (t,x),p, X) \leq -\gamma \eta \varphi(x) < 0
$$
whenever $(t,x) \in (0,T) \times \Omega$ and
$
u_i(t,x) > \M_i u(t,x)
$
(since $u$ is a subsolution).
However, $- \eta \varphi(x) - C_\eta$ is independent of $i$ and thus
\begin{align*}% \label{eq:obstacleequiv}
u_i(t,x) > \M_i u(t,x) &\iff u^\eta_i(t,x) > \M_i u^\eta(t,x)
\end{align*}
and we can conclude that
$$
\max \{a+ F_i(t,x,u_i^\eta(t,x),p, X), u^\eta_i - \M_i u^\eta(t,x) \}  \leq 0
$$
whenever $(a,p,X)\in \supjetbar u^\eta_i(t,x)$, $(t,x) \in (0,T) \times \Omega$.
On the other hand, if $x \in \partial \Omega$, we have by \eqref{ass:BC1}
\begin{align*}
B_i(t, x, u^\eta_i(t,x), p) =& B_i(t,x,u_i^\eta(t,x), p + \eta D \varphi(x)) - \eta \langle n(x) , D\varphi(x) \rangle  \\
\leq & B_i(t,x, u(t,x), p + \eta D\varphi(x)) - \eta
\end{align*}
and thus $u^\eta$ is a subsolution to \eqref{problem} with the boundary condition
$$
B_i(t,x,r,p)  :=\langle n(x),p\rangle + f_i(t,x,r)  \leq 0 \quad \mbox{replaced by} \quad \check B_i(t,x,r,p) :=B_i(t,x,r,p)  + \eta \leq 0.
$$
A similar calculation shows that $v^\eta$ is a supersolution to \eqref{problem} with boundary condition $\hat B_i(t,x,r,u): = B_i(t,x,r,p) - \eta \geq 0$.
Consequently, it suffices to prove the comparison $u^\eta \leq v^\eta$ when $u^\eta$ and $v^\eta$ are sub- and supersolutions to \eqref{problem} with boundary conditions $\check B_i(t,x,r,p)$ and $\hat B_i(t,x,r,p)$, respectively, since we retrieve our result in the limit as $\eta \to 0$.

\subsubsection*{Step 2. Avoiding the Neumann boundary condition.}

We now construct a test function which allows us to discard the Neumann boundary condition in \eqref{problem}. We will need the following comparison principle without boundary condition whose proof, which is similar to the current, is postponed to the end of the section.

\begin{proposition} \label{prop:comparisonnoboundary}
Let $u$ and $v$ be viscosity sub- and supersolutions, respectively, to
\begin{align*}
&\max\{{\partial_t u_i} + F_i(t,x,u_i,Du_i,D^2u_i) , u_i(t,x)-\M_iu(t,x) \}=0 \\
&u_i(0,x) = g_i(x)
\end{align*}
on $[0,T) \times \bar \Omega$ in the sense of Definition \ref{def:solution} $(i)$ and $(iii)$.
Then,
$$
 \sup_{[0,T) \times \bar \Omega} (u_i -v_i) \leq \max_{k \in\{1,\dots, m\}} \sup_{((0,T) \times \partial \Omega) \cup (\{0\} \times \bar \Omega)}(u_k-v_k)^+.
$$
\end{proposition}

Let us now assume the opposite of what we seek to prove, i.e., that there exists a non-empty set $\I \subset \{1,\dots, m\}$ and $(\hat t, \hat x) \in [0,T] \times \bar \Omega$ such that
\begin{align}\label{eq:maxpointdelta}
\max_{k \in \{1,\dots, m\}}\sup_{(0,T)\times\Omega} (u_k -v_k) = u_i(\hat t, \hat x) - v_i(\hat t, \hat x) = \delta >0
\end{align}
for any $i\in \I$. Since $u_i \leq g_i \leq v_i$ on $\{0\} \times \bar \Omega$ by definition, we have $\hat t >0$ and by Proposition \ref{prop:comparisonnoboundary} we can then also assume $\hat x \in \partial \Omega$. Moreover, if we set
$$
u^\theta(t,x): = u(t,x) - \frac{\theta}{T-t}
$$
for $\theta >0$ arbitrary, we have $u^\theta < u$,
$$
(a, p, X ) \in \supjetbar u^\theta (t,x) \iff (a +\frac{\theta}{(T-t)^2}, p, X) \in \supjetbar	u(t,x),
$$
and
\begin{align*}
(i) &\, u_i(t,x) > \M_i u(t,x) \iff  u^\theta_i(t,x) > \M_i  u^\theta(t,x) \quad \mbox{(since $u^\theta_i - u_i$ is independent of $i$)}, \notag \\
(ii) &\, \langle n(x), p \rangle + f_i(t,x,u_i^\theta(x)) \leq \langle n(x), p \rangle + f_i(x,u_i(x)) \quad \mbox{(by \eqref{ass:BC1})} \notag \\
(iii) &\,  a + F_i(t,x,u_i^\theta(t,x),p, X) \leq a+ F_i(t,x,u_i(t,x),p, X) \quad \mbox{(by \eqref{ass:F1})}.
\end{align*}
It follows immediately that $u^\theta$ is a subsolution to \eqref{problem}  with $F_i$ replaced by
\begin{equation*}
F_i^\theta (t,x,r,p,X)= F_i(t,x,r,p,X)+ \frac{\theta}{T-t}
\end{equation*}
 and that $u^\theta \to -\infty$ as $t \to T$. We may therefore also assume that $\hat t < T$, %and that $F(t,x,r,p,X) < -c$ for some $c>0$ whenever $F(t,x,r,p,X) \leq 0$,
since if not, we can prove $u^\theta \leq v$ as follows and then retrieve our result in the limit as $\theta \to 0$.

For $\epsilon >0$ (tending to $0$ in the final argument) and $i \in \mathcal I$ arbitrary but fixed, let
\begin{equation*}
\varphi_\epsilon (t,x,y) = \frac{1}{2\epsilon} |x-y|^2 + |x-\hat x|^4 + |y-\hat x|^4 + |t-\hat t|^2 - f_i(\hat t, \hat x, u_i(\hat t, \hat x)) \langle n(\hat x), x-y \rangle
\end{equation*}
and consider the function
\begin{equation*}
\Phi_\epsilon(t,x,y)= u_i(t,x)-v_i(t,y)-  \varphi_\epsilon(t,x,y)
\end{equation*}
which is USC by construction. Let $(t_\e, x_\e, y_\e)$ be the maximum point of $\Phi_\e$ on $[0,T) \times \bar \Omega \times \bar \Omega$. (This maximum point exists for $\e$ small as $\Phi_\e$ is USC, $\hat t <T$ and $\bar \Omega$ is compact.) Clearly, $|x_\e - y_\e| \to 0$ as $\e \to 0$ as $(t_\e, x_\e, y_\e)$ is a maximum point. Since
$$
2\Phi_\e (t_\e, x_\e, y_\e ) \geq \Phi_\e (t_\e, x_\e, x_\e) + \Phi_\e (t_\e, y_\e, y_\e)
$$
we have
\begin{align*}
\frac{1}{\e}|x_\e -y_\e|^2 - 2 f_i(\hat t, \hat x, u(\hat t, \hat x)) \langle n(\hat x), x_\e-y_\e  \rangle \leq u_i(t_\e, x_\e)-  u_i(t_\e, y_\e)  +  v_i(t_\e, x_\e) - v_i(t_\e, y_\e) < \infty
\end{align*}
where the last inequality holds since $u_i-v_i$ is USC. This gives that also
$\frac{1}{\e}|x_\e -y_\e|^2 \to 0$ as $\e\to 0$
and $(t_\e,x_\e) \to (\hat t, \hat x)$ since $u_i(t,x) - v_i(t,x) - \varphi_\e(t,x,x)$ has a maximum at $(\hat t, \hat x)$. Moreover, from the upper- and lower semi-continuity of $u$ and $v$ we get
\begin{equation*}
u_i(t_\e, x_\e) \to u_i(\hat t, \hat x) \quad \mbox{and} \quad v_i(t_\e, x_\e) \to v_i(\hat t, \hat x).
\end{equation*}
In particular, from the definition of $\varphi_\e$ and $(t_\e, x_\e, y_\e)$ we get
\begin{align}\label{eq:usclimit}
u_i (\hat t, \hat x) - v_i (\hat t, \hat x) =& \Phi_\e(\hat t, \hat x,\hat x) \leq \Phi_\e(t_\e, x_\e, y_\e) \notag \\
\leq& u_i (t_\e, x_\e) - v_i (t_\e, x_\e) + f_i(\hat t, \hat x, u_i(\hat t, \hat x))\langle n(\hat x), x_\e-y_\e\rangle.
\end{align}
Note that $u_i \in USC ([0,T]\times\bar \Omega)$ so we have $\limsup _{\e \to 0} u_i(t_\e, x_\e) \leq u_i(\hat t, \hat x)$.
Since $(t_\e, x_\e, y_\e) \to (\hat t, \hat x, \hat x)$ as $\e \to 0$, this inequality cannot be strict; if it were,
\eqref{eq:usclimit} shows that we must necessarily have $\liminf_{\e \to 0} v_i(t_\e,x_\e)< v(\hat t, \hat x)$ as well,
but this contradicts $v_i \in LSC([0,T]\times\bar \Omega)$.
An analogous argument shows $v_i(t_\e, x_\e)  \to v_i(\hat t, \hat x)$.

We now invoke the exterior sphere condition which implies the existence of $r>0$ s.t.
$$
\langle n(x_\e) , x_\e - \hat x\rangle > -\frac{1}{2r}|\hat x - x_\e|^2		  \quad \mbox{for any $\hat x \in \bar \Omega$ and $x_\e \in \partial \Omega$}.
$$
Differentiating gives
\begin{align*}
D_x\varphi_ \e(t,x,y)= \frac{1}{\e} (x-y) + 4|x -\hat x|^2 (x-\hat x) -f_i(\hat t, \hat x, u_i(\hat t, \hat x)) n(\hat x)
\end{align*}
and thus, if $x_\e \in \partial \Omega$ we have
\begin{align*}
&B_i(t_\e, x_\e, u_i(t_\e, x_\e), D_x \varphi_\e(t_\e, x_\e, y_\e)) = \langle n(x_\e), D_x\varphi_\e(t_\e, x_\e, y_\e ) \rangle +  f_i(t_\e, x_\e, u_i(t_\e, x_\e)) \\
=&\left\langle n(x_\e), \frac{1}{\e} (x_\e-y_\e) + 4 |x_\e- \hat x|^2 (x_\e-\hat x) - f_i(\hat t, \hat x, u_i(\hat t, \hat x)) n(\hat x)\right\rangle  + f_i(t_\e, x_\e, u_i(t_\e, x_\e)) \\
\geq&  -\frac{1}{2r \e}|x_\e -y_\e|^2 + 4|x_\e -\hat x|^2 \langle n(x_\e), x_\e-\hat x\rangle - f_i(\hat t, \hat x, u_i(\hat t, \hat x)) \langle n(x_\e), n(\hat x) \rangle + f_i(t_\e, x_\e, u_i(t_\e, x_\e)).
\end{align*}
Hence, as $\e \to 0$ we have
\begin{equation} \label{eq:BCcond1}
B_i(t_\e, x_\e, u_i(t_\e, x_\e), D_x\varphi_\e(t_\e, x_\e, y_\e))  \to D %\geq o(1).
\end{equation}
for some $D\geq 0$. Similarly, for the supersolution $v$ we get for $y_\e \in \partial \Omega$
\begin{align*}
&B_i(t_\e, y_\e, v_i(t_\e, y_\e), -D_y \varphi_\e(t_\e, x_\e, y_\e)) = \langle n(y_\e), -D_y\varphi_\e(t_\e, x_\e, y_\e ) \rangle +  f_i(t_\e, y_\e, v_i(t_\e, y_\e)) \\
=&\left\langle n(y_\e), \frac{1}{\e} (x_\e-y_\e)  - 4|y_\e -  \hat y|^2 (y_\e-\hat x) - f_i(\hat t, \hat x, u_i(\hat t, \hat x)) n(\hat x)\right\rangle + f_i(t_\e, y_\e, v_i(t_\e, y_\e)) \\
\leq & \frac{1}{r\e}|x_\e - y_\e|^2 -4 |y_\e -\hat x|^2 \langle n(x_\e),  y_\e-\hat x \rangle-f_i(\hat t, \hat x, u_i(\hat t, \hat x)) \langle n(y_\e), n(\hat x) \rangle + f_i(t_\e, y_\e, v_i(t_\e, y_\e))
\end{align*}
and thus that, as $\e \to 0$,
\begin{equation} \label{eq:BCcond2}
B_i(t_\e, y_\e, v_i(t_\e, y_\e), -D_y \varphi_\e(t_\e, x_\e, y_\e) \to \tilde D % \leq o(1)
\end{equation}
for some $\tilde D \leq0$ (since $u_i > v_i$ at $(\hat t, \hat x)$ and $f_i$ is non-decreasing by assumption \eqref{ass:BC1}).
Hence, if
$$
u_i(\hat t, \hat x) -v_i(\hat t, \hat x)=\delta>0
$$
is a positive maximum of $u_i-v_i$ over $(0,T) \times \bar \Omega$ we must have
\begin{align*}
&\mbox{$\min \left\{ a + {F_i}\left(t_\e,x_\e,u_i(t_\e,x_\e), D_x \varphi_\e(t_\e, x_\e,y_\e), X \right),  u_{i}(t_\e,x_\e) - \M_i u(t_\e,x_\e) \right\} \leq 0$} \\
&\mbox{for $(a,D_x \varphi_\e, X) \in \supjetbar u_i(t_\e, x_\e)$, \quad and}  \\
&\mbox{$\min \left\{ \tilde a + {F_i}\left(t_\e,y_\e,v_i(t_\e,y_\e), -D_y \varphi_\e(t_\e, x_\e,y_\e), Y \right),  v_{i}(t_\e,x_\e) - \M_i v(t_\e,x_\e) \right\} \geq 0$}\\
&\mbox{for $(\tilde a,-D_y \varphi_\e, Y) \in \subjetbar v_i(t_\e, x_\e)$}
\end{align*}
provided $\e$ is small enough.
Indeed, this holds by continuity of $F_i$ and since, by Step 1, we can consider $u_i$ and $v_i$ to be sub- and supersolutions to \eqref{problem} with boundary conditions $\check B_i : =B_i + \eta\leq 0$ and  $\hat B_i : =B_i -\eta \geq 0 $, respectively. Thus, since we have \eqref{eq:BCcond1} and \eqref{eq:BCcond2}, the boundary conditions cannot be satisfied and therefore the equation must hold on the boundary (in the sub- /supersolution sense).

\subsubsection*{Step 3. Avoiding the obstacle.}

We will now argue as in Ishii-Koike \cite{IK91} to ensure that the subsolution (more precisely, at least one component of it) lies strictly above its obstacle at $(\hat t, \hat x)$. To do this, recall the set $\I$  and assume that $u_i(\hat t, \hat x) \leq \mathcal M_i u(\hat t, \hat x)$, i.e.,
$$
u_i(\hat t, \hat x) \leq \max_{i \neq j}\{ {u_j(\hat t, \hat x) -c_{ij}(\hat t, \hat x)}\},
$$
whenever $i \in \I$. This implies the existence of $k \in \{1, \dots, i-1, i+1,\dots, m\}$ s.t.
$$
u_i(\hat t, \hat x) + c_{ik}(\hat t, \hat x) \leq u_k(\hat t, \hat x)
$$
Moreover, since $v$ is a supersolution we have
$$
v_i(\hat t, \hat x) \geq v_k(\hat t, \hat x) -c_{ik}(\hat t, \hat x).
$$
Combining the above two inequalities yield
$$
u_i(\hat t, \hat x) -v_i(\hat t, \hat x) \leq  u_k(\hat t, \hat x) -v_k (\hat t, \hat x)
$$
but since $(\hat t, \hat x)$ is a maximum of $u_i-v_i$ and $i \in \I$ this must in fact be an equality and $k \in I$ as well. Repeating this as many times as necessary, we find the existence of a sequence of indices $\{i_1, i_2, \dots, i_p, i_1\}$, $i_p \neq i_{p+1}$ such that
$$
u_{i_1}(\hat t, \hat x) + c_{i_1 i_2}(\hat t, \hat x) + c_{i_2 i_3}(\hat t, \hat x)  + \dots c_{i_p i_1}(\hat t, \hat x) \leq u_{i_1}(\hat t, \hat x)
$$
which implies
$$c_{i_1 i_2}(\hat t, \hat x) + c_{i_2 i_3}(\hat t, \hat x)  + \dots c_{i_p i_1}(\hat t, \hat x) \leq 0,
$$ a contradiction to \eqref{ass:O1} $(iii)$. Thus, there exists at least one index $i \in \I$ s.t.
$$
u_i (\hat t, \hat x) > \M_i u(\hat t, \hat x).
$$
Since $u_i(x_\e, t_\e) \to u_i(\hat t, \hat x)$ and $v_i(x_\e, t_\e)\to v_i(\hat t, \hat x)$ for all $i \in \{1, \dots, m\}$, we can then conclude from this and Step 2 that
\begin{align}\label{eq:liftedandinterior}
&\mbox{$ a + {F_i}\left(t_\e,x_\e,u_i(t_\e,x_\e), D_x \varphi_\e(t_\e, x_\e,y_\e), X \right) \leq 0$ for $(a,D_x \varphi_\e, X) \in \supjetbar u_i(t_\e, x_\e)$, \quad and} \notag  \\
&\mbox{$\tilde a + {F_i}\left(t_\e,y_\e,v_i(t_\e,y_\e), -D_y \varphi_\e(t_\e, x_\e,y_\e), Y \right) \geq 0$ for $(\tilde a,-D_y \varphi_\e, Y) \in \subjetbar v_i(t_\e, y_\e)$}
\end{align}
for at least one $i \in \I$ and $\e$ small enough.

\subsubsection*{Step 4. Reaching the contradiction}

We are now ready to reach our final contradiction. To do this, we will use the following lemma, the so called maximum principle for semi-continuous functions. Lemma \ref{lemma:thm8.3} corresponds to Theorem 8.3 of \cite{CIL92} in a less general but for our purposes sufficient form.
\begin{lemma} \label{lemma:thm8.3}
Suppose that $(t_\e, x_\e, y_\e)$ is a maximum point of
$$
u_i(t,x)  -v_i(t,y) -\varphi_\e(t,x,y)
$$
over $(0,T) \times \bar \Omega$. Then, for each $\theta >0$ there are $X, Y \in \S(n)$ such that
\begin{align*}
(i) &\,(a, D_x \varphi_\e(t_\e, x_\e, y_\e), X) \in \supjetbar u_i(t\e, x_\e) \quad\mbox{and}\quad \\
&\, (\tilde a, -D_y \varphi_\e(t_\e,x _\e, y_\e), Y) \in \subjetbar v_i(t_\e, y_\e), \\% \quad (-\tilde a, D_y, Y) \supjetbar\\
(ii)&\,
\left (
\begin{matrix}
X&0\\
0&-Y
\end{matrix} \right )
\leq A+ \theta A^2, \\
(iii)& \,a- \tilde a = \frac{\partial}{\partial t} \varphi_\e(t_\e, x_\e, y_\e),
\end{align*}
where $A :=D_x^2 \varphi_\e(t_\e, x_\e, y_\e)$ is the Hessian matrix of $\varphi(t,x,y)$ (w.r.t. $x$ and $y$).
\end{lemma}

We first note that
\begin{align*}
\frac{\partial}{\partial t} \varphi_\e(t_\e, x_\e, x_\e) =& 2  (t_\e-\hat t)\\
D_x\varphi_\e (t_\e, x_\e, y_\e) =  & \frac{1}{\e} (x_\e - y_\e) + 4 |x_\e -\hat x|^2(x_\e -\hat x) - f_i(\hat t, \hat x, u_i(\hat t, \hat x)) n(\hat x) \\
-D_y \varphi_\e(t_\e,x_\e, y_\e) = & \frac{1}{\e}(x_\e - y_\e) - 4 |y_\e -\hat x|^2 (y_\e-\hat x) - f_i(\hat t, \hat x, u_i(\hat t, \hat x)) n(\hat x) \\
A:= D^2_x \varphi_\e(t_\e,x_\e, y_\e) = & \frac{1}{\e}
\left (\begin{matrix}
I &-I \\
-I& I
\end{matrix} \right)
+ \O(|x_\e - \hat x|^2 + |y_\e - \hat x|^2)
\end{align*}
which gives
$$
A^2 = \frac{2}{\e^2} \left (\begin{matrix}
I &-I \\
-I& I
\end{matrix} \right) + \O\left (\frac{1}{\e}(|x_\e -\hat x|^2 +|y_\e-\hat y|^2) + |x_\e - \hat x|^4 + |y_\e - \hat x|^4 \right ).
$$
Choosing $\theta = \e$ in Lemma \ref{lemma:thm8.3}  $(ii)$ gives the existence of
$$
(a, D_x \varphi_\e(t_\e,x_\e, x_\e), X) \in \supjetbar u_i(t_\e, x_\e) \quad\mbox{and}\quad (\tilde a, -D_y \varphi_\e(t_\e, x_\e,x_\e), Y) \in \subjetbar v_i( t_\e, x_\e),
$$
with
\begin{align*} %\label{eq:XYineq}
\left (
\begin{matrix}
X&0\\
0&-Y
\end{matrix} \right ) \leq & \frac{3}{\e}
\left (\begin{matrix}
I &-I \\
-I& I
\end{matrix} \right)
+ \O(|x_\e - \hat x|^2 + |y_\e - \hat x|^2).
\end{align*}
Note in particular that this implies that for any $\xi >0$ fixed there exists $\e>0$ such that the above holds with
\begin{align} \label{eq:XYineq}
\left (
\begin{matrix}
X&0\\
0&-Y
\end{matrix} \right ) \leq  \frac{3}{\e}
\left (\begin{matrix}
I &-I \\
-I& I
\end{matrix} \right)
+  \xi   \left (\begin{matrix}
I&0\\
0&I
\end{matrix} \right ) \iff \left (
\begin{matrix}
X-\xi I&0\\
0&-(Y+\xi I)
\end{matrix} \right ) \leq  \frac{3}{\e}
\left (\begin{matrix}
I &-I \\
-I& I
\end{matrix} \right).
\end{align}
Moreover, we have from \eqref{eq:liftedandinterior} that
$$
a + {F_i}\left(t_\e,x_\e,u_i(t_\e,x_\e), D_x \varphi_\e(t_\e, x_\e,y_\e), X \right) \leq 0 \quad \mbox{and} \quad \tilde a + {F_i}\left(t_\e,x_\e,v_i(t_\e,y_\e), -D_y \varphi_\e(t_\e, x_\e), Y \right) \geq 0
$$
which implies
\begin{equation} \label{eq:contradict}
2  (t_\e-\hat t) = a-\tilde a
\leq {F_i}\left(t_\e,y_\e,v_i(t_\e,y_\e), -D_y \varphi_\e(t_\e,x_\e,y_\e), Y \right)  -  {F_i}\left(t_\e,x_\e,u_i(t_\e,x_\e), D_x \varphi_\e(t_\e, x_\e,y_\e), X \right)
\end{equation}
by Lemma \ref{lemma:thm8.3} $(iii)$.

What remains is to show that \eqref{eq:contradict} is inconsistent with \eqref{eq:maxpointdelta}, \eqref{eq:XYineq} and the assumptions \eqref{ass:F1} - \eqref{ass:F3}. To see this, note that
\begin{align}	\label{eq:firstineq}
&{F_i}\left(t_\e,y_\e,v_i(t_\e,y_\e), -D_y \varphi_\e(t_\e,x_\e,y_\e), Y \right)  -  {F_i}\left(t_\e,x_\e,u_i(t_\e,x_\e), D_x \varphi_\e(t_\e, x_\e,y_\e), X \right) \notag \\
\leq & {F_i}\left(t_\e,y_\e,v_i(t_\e,y_\e), -D_y \varphi_\e(t_\e,x_\e,y_\e), Y \right)  -  {F_i}\left(t_\e,x_\e,u_i(t_\e,x_\e), -D_y \varphi_\e(t_\e, x_\e,y_\e), X \right) \notag \\
& + \omega(4(|\hat x-x_\e|^3 + |\hat x- y_\e|^3)) \notag \\
\leq & {F_i}\left(t_\e,y_\e,u_i(t_\e,x_\e), -D_y \varphi_\e(t_\e,x_\e,y_\e), Y \right)  -  {F_i}\left(t_\e,x_\e,u_i(t_\e,x_\e), -D_y \varphi_\e(t_\e, x_\e,y_\e), X \right) \notag \\
&+ \omega(4(|\hat x-x_\e|^3 + |\hat x- y_\e|^3)) - \gamma(u_i(t_\e, x_\e) - v_i(t_\e,y_\e))
\end{align}
where we have used, in turn, \eqref{ass:F2},
$$|D_x  \varphi_\e(t_\e,x_\e, y_\e)- (-D_y \varphi_\e(t_\e,x_\e, y_\e))|=  4(|\hat x-x_\e|^2(x_\e-\hat x)+ |\hat x- y_\e|^2 (y_\e- \hat x)),$$
and \eqref{ass:F1} (more precisely, \eqref{eq:F1star}).
Since \eqref{eq:XYineq} holds we get from \eqref{ass:F3} that
\begin{align*}
F_i(t_\e,y_\e, r, p, Y+ \xi I) - F_i(t_\e, x_\e, r, p, X-\xi I) \leq \omega(\frac{1}{\e} |x_\e-y_\e|^2 + |x_\e-y_\e|(|p|+1))
\end{align*}
and from \eqref{ass:F2} that
\begin{align*}
F_i(t_\e,y_\e, r, p, Y+ \xi I) - F_i(t_\e, x_\e, r, p, X-\xi I) \geq F_i(t_\e,y_\e, r, p, Y) - F_i(t_\e, x_\e, r, p, X) - 2 \omega(\xi)
\end{align*}
and by combining these displays we have
\begin{equation} \label{eq:secondineq}	
F_i(t_\e,y_\e, r, p, Y) - F_i(t_\e, x_\e, r, p, X) \leq \omega(\frac{1}{\e} |x_\e-y_\e|^2 + |x_\e-y_\e|(|p|+1)) + 2 \omega (\xi).
\end{equation}
Putting \eqref{eq:contradict}, \eqref{eq:firstineq}, and \eqref{eq:secondineq} together gives
\begin{align*}
2  (t_\e-\hat t)  \leq  &\, \omega(4(|\hat x-x_\e|^3 + |\hat x- y_\e|^3)) - \gamma(u_i(t_\e, x_\e) - v_i(t_\e,y_\e)) \\
& + 2 \omega (\xi) + \omega(\frac{1}{\e} |x_\e-y_\e|^2 + |x_\e-y_\e|(|-D_y \varphi_\e(t_\e, x_\e, y_\e)|+1)).
\end{align*}
Taking the limit as $\e \to 0$ and recalling $\frac{1}{\e}|x_\e -y_\e|^2 \to 0$ we arrive at
\begin{equation*}
\gamma \delta \leq 2 \omega (\xi)
\end{equation*}
which is a contradiction since $\xi>0$ was arbitrary and $\gamma \delta > 0$.
$\, {}_\blacksquare$

\noindent \newline
\begin{proof}[Proof of Proposition \ref{prop:comparisonnoboundary}]
Assume first that $u_i \leq v_i$ on $(0,T) \times \partial \Omega$ for all $i \in \{1,\dots, m\}$. As above, assume the existence of $\I$ and $(\hat t, \hat x)$ such that
$$
0< \delta = u_i(\hat t, \hat x)  -v_i(\hat t, \hat x)
$$
for all $i \in \I$. Now, let
\begin{equation*}
\varphi_\epsilon (t,x,y) = \frac{1}{2\epsilon} |x-y|^2 + |x-\hat x|^4 + |y-\hat x|^4 + |t-\hat t|^2
\end{equation*}
and construct $\Phi_\e (t,x,y) := u_i(t,x)  -v_i (t,y) - \varphi_\e(t,x,y)$.
Calculations analogous to those in the first half of Step 2 above shows
$$
\frac{1}{\e} |x_\e-y_\e|\to 0, \quad x_\e \to  \hat x, \quad u_i(t_\e,x_\e) \to u_i (\hat t, \hat x), \quad \mbox{and} \quad v_i(t_\e, y_\e) \to v_i(\hat t, \hat x)
$$
and we can thus simply repeat Step 3-4 above to get the desired contradiction.

For the general case, let
$$
K :=\max_{i \in\{1,\dots, m\}} \sup_{((0,T) \times \partial \Omega) \cup (\{0\} \times \bar \Omega)} (u_i-v_i)^+ \geq 0
$$
and note that $\tilde u:=u-K \leq u$ is a subsolution to \eqref{problem} in the sense of Definition \ref{def:solution} $(i)$ and $(iii)$ since $\tilde u_i(0,\cdot) \leq u_i(0,\cdot) \leq g_i$,  $K$ is independent of $i$ and
$$
a+ F(t,x,\tilde u_i, p, X) \leq a + F(t,x, u_i, p, X)
$$
by \eqref{ass:F1}. Moreover, $\tilde u \leq v$ on $(0,T) \times \partial \Omega$ so we can apply the above result to conclude that
$$
 \tilde u = u- K \leq v  \iff u-v \leq K
$$
and we are done.
\end{proof}

\noindent \newline
\begin{proof}[Proof of Corollary \ref{corr:mixedcomparison}]
If $u$ is a
viscosity subsolution, then so is $u-K$ for all $K>0$. It thus suffices to
prove that if $u\leq v$ on $((0,T) \times \partial \Omega )\setminus G)$, then $u\leq v$ in $[0,T) \times \bar \Omega$. If $G=(0,T) \times \partial \Omega$, this implication and its proof is identical to Theorem \ref%
{thm:comparison} and if $G= \emptyset$ it is identical to Proposition \ref{prop:comparisonnoboundary}. If $G \subset (0,T) \times \partial \Omega $ is a non-empty proper subset, we know by
assumption that $u\leq v$ on $((0,T) \times \partial \Omega ) \setminus G$ and so the maximum point $(\hat t, \hat x)$ defined in \eqref{eq:maxpointdelta} must belong to the set $G$ where the boundary condition is satisfied. Hence, we can follow the proof of Theorem \ref{thm:comparison} to conclude that $u\leq v$ in $[0,T) \times \bar \Omega$.
\end{proof}

%%%%%%%%%%%%%%%%%%%%%%%%%%%%%%%%%%%%%%%%%%%%%
%%%%%%%%%%%%%%%%%%%%%%%%%%%%%%%%%%%%%%%%%%%%%
%%%%%%%%%%%%%%%%%%%%%%%%%%%%%%%%%%%%%%%%%%%%%
%%%%%%%%%%%%%%%%%%%%%%%%%%%%%%%%%%%%%%%%%%%%%
%%%%%%%%%%%%%%%%%%%%%%%%%%%%%%%%%%%%%%%%%%%%%
%%%%%%%%%%%%%%%%%%%%%%%%%%%%%%%%%%%%%%%%%%%%%
%%%%%%%%%%%%%%%%%%%%%%%%%%%%%%%%%%%%%%%%%%%%%
%%%%%%%%%%%%%%%%%%%%%%%%%%%%%%%%%%%%%%%%%%%%%
%%%%%%%%%%%%%%%%%%%%%%%%%%%%%%%%%%%%%%%%%%%%%
%%%%%%%%%%%%%%%%%%%%%%%%%%%%%%%%%%%%%%%%%%%%%
%%%%%%%%%%%%%%%%%%%%%%%%%%%%%%%%%%%%%%%%%%%%%
%%%%%%%%%%%%%%%%%%%%%%%%%%%%%%%%%%%%%%%%%%%%%
%%%%%%%%%%%%%%%%%%%%%%%%%%%%%%%%%%%%%%%%%%%%%
%%%%%%%%%%%%%%%%%%%%%%%%%%%%%%%%%%%%%%%%%%%%%
%%%%%%%%%%%%%%%%%%%%%%%%%%%%%%%%%%%%%%%%%%%%%
%%%%%%%%%%%%%%%%%%%%%%%%%%%%%%%%%%%%%%%%%%%%%
%%%%%%%%%%%%%%%%%%%%%%%%%%%%%%%%%%%%%%%%%%%%%
%%%%%%%%%%%%%%%%%%%%%%%%%%%%%%%%%%%%%%%%%%%%%

\setcounter{equation}{0} \setcounter{theorem}{0}

\section{Proof of Theorem \ref{thm:existence}: Perron's method and barrier construction.} %\label{sec:proofexistence}

Throughout the section, we assume that the assumptions of Theorem \ref{thm:existence} hold, i.e., that \eqref{ass:D1}, \eqref{ass:F1} - \eqref{ass:F3}, \eqref{ass:BC1}, \eqref{ass:O1} - \eqref{ass:O3}, and \eqref{ass:compatibility} hold. Our proof of existence follows the machinery of Perron's method. In particular, we have the following result.
\begin{proposition} \label{prop:Perron}
Assume that, for each $i \in \{1, \dots, m\}$ and $\hat x \in \bar \Omega$ there exist families of continuous viscosity sub- and supersolutions, $\{u^{i,\hat x,\epsilon}\}_{\epsilon >0}$ and $\{v^{i,\hat x,\epsilon}\}_{\epsilon >0}$, to \eqref{problem} such that
$$
\sup_{\epsilon} u^{i,\hat x,\epsilon}_i(0,\hat x) = g_i(\hat x) = \inf_{\epsilon} v^{i,\hat x,\epsilon}_i(0,\hat x).
$$
Then,
$$
w(t,x) =\sup \{u(t,x)\,: \,\mbox{$u$ is a subsolution to \eqref{problem}} \}
$$
is a viscosity solution to \eqref{problem}.
\end{proposition}
With this result given, what remains is to construct appropriate barriers, i.e., families of viscosity sub- and supersolutions taking on the correct initial data. More specifically, we prove the following.

\begin{proposition}\label{prop:barriers}
For any $\hat x \in \bar \Omega$ and $\epsilon > 0$, there exist non-negative constants $A, B, C$ and $\kappa$ such that $U^{\hat x, \epsilon} :=(U_1^{\hat x, \epsilon}, \ldots, U^{\hat x, \epsilon}_m)$ and $V^{i, \hat x, \epsilon}:=(V_1^{i, \hat x, \epsilon}, \dots, V^{i, \hat x, \epsilon}_m)$,
\begin{align*}
U^{\hat x,\epsilon}_j(t,x)&=g_j(\hat x) - A(\varphi(x) - \varphi(\hat x)) - B \exp(\kappa \varphi(x)) |x-\hat x|^2- \epsilon - Ct, \\
V^{i,\hat x,\epsilon}_j(t,x) &= g_i(\hat x)  + A(\varphi(x) - \varphi(\hat x)) +B \exp(\kappa \varphi(x)) |x-\hat x|^2 + \epsilon + Ct + c_{ij} (t,x),
\end{align*}
where $\varphi(x)$ is given in Lemma \ref{lemma:varphi} and $i\in\{1,\dots, m\}$, are viscosity sub- and supersolutions to \eqref{problem}, respectively. Moreover,
$$
\sup_{\epsilon} U^{\hat x,\epsilon}_i(0,\hat x) = g_i(\hat x) = \inf_{\epsilon} V^{i,\hat x,\epsilon}_i(0,\hat x).
$$
\end{proposition}
Combining Propositions \ref{prop:Perron} and \ref{prop:barriers} above proves the existence part of Theorem \ref{thm:existence}. Uniqueness follows immediately from Theorem \ref{thm:comparison} and the definition of sub- and supersolutions. What needs to be done is thus to prove the above propositions. Being the non-standard one, we start with Proposition \ref{prop:barriers}.

\noindent \newline
\begin{proof}[Proof of Proposition \ref{prop:barriers}]
We prove only the supersolution property (the subsolution property is proven analogously but without the need to deal with the obstacle as in \eqref{eq:obstacledealing} below).  We need to show that $(V_1^{i, \hat x, \epsilon}, \cdots, V^{i, \hat x, \epsilon}_m)$ satisfies condition $(i)-(iii)$ of Definition \ref{def:solution}. To ease notation, we suppress the superindices $i, \hat x, \epsilon$ and write $V_j$ in place of $V^{i,\hat x,\epsilon}_j$.

We begin with condition $(iii)$. For any $\epsilon >0$ and $A$ given, we can ensure $V_j(0,x) \geq g_j(x), \forall x \in \bar \Omega$ and all $j \in \{1,\dots,m\}$ by choosing $B$ sufficiently large.
Indeed, this is possible since $\varphi, g$, and $c_{ij}$ are continuous and
$$
V_j(0,\hat x) = g_i(\hat x) + \epsilon + c_{ij}(0,\hat x) \geq g_j(\hat x),
$$
where the last inequality is due to Assumption \eqref{ass:compatibility}.
Moreover, we can let $\epsilon \to 0$ (by letting $B \to \infty $ if necessary) and since $c_{ii}(t,x)\equiv 0$ we thus have that
\begin{equation} \label{eq:initialensurance}
\inf_{\epsilon} V_i(0, \hat x) = g_i(\hat x).
\end{equation}
Note that the choice of $B=B(\epsilon,A)$ can be made with $C= \kappa = 0$; later increasing $\kappa$ and/or $C$ will only increase $V$ further while keeping \eqref{eq:initialensurance}.

We next turn to condition $(i)$ of Definition \ref{def:solution}.
Starting with the obstacle, we have
\begin{align} \label{eq:obstacledealing}
&V_j(t,x) - \mathcal M_j V(t,x) = V_j(t,x) - \max_{j \neq k} \{V_k (t,x) -c_{jk}(t,x) \} \notag \\
=& \,c_{ij}(t,x) - \max_{j \neq k} \{c_{ik}(t,x) - c_{jk}(t,x)\} = c_{ij}(t,x) - c_{i\hat k}(t,x) + c_{j\hat k}(t,x) \geq  0
\end{align}
where the last inequality is by assumption \eqref{ass:O3}.
Concerning the second part we observe that, for some $C$ large enough, it holds that
\begin{equation} \label{eq:operatorpart}
\partial_t V_i(t,x)+ F_i(t,x,V_i(t,x), D V_i(t,x), D^2 V_i(t,x)) \geq 0
\end{equation}
for $(a,p,X) \in \subjet V_i(t,x)$, for all $(t,x) \in (0,T)\times\Omega$
and all $i \in \{1,\dots,m\}$.
Indeed, this follows after noticing that $V_i$ smooth and $F_i$ continuous give a lower bound for
$F_i(t,x,V_i(t,x), D V_i(t,x), D^2 V_i(t,x))$ on the compact region $[0,T]\times \bar \Omega$.
This lower bound can be made independent from $C$ (and $\epsilon$) since $F_i$ is assumed non-decreasing with $V_i$ and $V_i$ is non-decreasing in $C$ (and $\epsilon$); it may however depend on $A$, $B$ and the parameter $\kappa$ to be chosen later.
Since $\partial_t V_j(t,x) = C + \partial_t c_{ij}(t,x)$ and the latter term is bounded in $[0,T]\times \bar \Omega$,
we conclude that inequality \eqref{eq:operatorpart} holds for large enough $C=C(A,B,\kappa)$.

What remains to verify is condition $(ii)$ of Definition \ref{def:solution}.
Note that, for $x \in \partial \Omega$, satisfying \eqref{eq:operatorpart}
in the classical sense does not ensure ditto in the viscosity sense. We therefore instead focus on the
Neumann condition and intend to prove that
\begin{equation*} %\label{eq:boundarycond}
\langle n(x), D V_j(t,x) \rangle + f_j(t, x,V_j(t,x)) \geq 0
\end{equation*}
for $(t,x) \in (0,T)\times\partial \Omega$ and for all $j \in \{1,\dots,m\}$.
Recall that we have chosen $B$ s.t. $V_j(0,x)  \geq g_j(x)$ in $\bar \Omega$.
Since $\partial_t c_{ij}(t,x)$ is bounded from below,
we can also ensure that
$
V_j(t,x) \geq g_j(x)
$
holds for all $(t,x) \in [0,T]\times \bar\Omega$ by increasing $C$ if necessary.
By the monotonicity property of $f$ (non-decreasing in $r$) we then have
$$
f_j(t,x,V_j(t,x)) \geq f_j(t,x,g_j(x)).
$$
Let $\tilde A$ be such that
$$
\min_{j \in \{1,\dots,m\}} \inf_{(t,x) \in [0,T]\times \partial \Omega}f_j(t,x,g_j(x))> -\tilde A.
$$
The Neumann boundary condition thus follows if we can show that
\begin{align}\label{tjohej_ny}
\langle n(x), D V_j(t,x) \rangle \geq \tilde A
\end{align}
for all $(t,x)\in (0,T)\times \partial \Omega$.
Differentiating $V_j(t,x)$ gives
$$
D V_j(t,x) = A D \varphi(x) + B \exp(\kappa \varphi(x) )  \left (2 (x-\hat x) +  \kappa  |x-\hat x|^2 D \varphi(x)  \right)+ D c_{ij}(t,x)
$$
and thus
\begin{align}\label{tjohej}
\langle n(x), D V_j(t,x) \rangle =&\,\langle n(x), A D \varphi(x) \rangle + \langle n(x), D c_{ij}(t,x)\rangle \notag \\
&+ B \exp(\kappa \varphi(x)) \langle n(x), 2 ( x-\hat x) + \kappa |x-\hat x|^2  D \varphi(x)\rangle
\end{align}
where $|\langle n(x), D c_{ij}\rangle|$ is bounded since $c_{ij}$ is smooth.
The exterior ball condition implies that
$$
\langle n( x) , x -\hat  x\rangle > - \frac{1}{2r}|\hat x - x|^2	 \quad \mbox{for any $\hat x \in \bar \Omega$ and $x \in \partial{\Omega}$},
$$
where $r$ depends on $\Omega$.
From this we see that, for $\kappa$ large enough and depending only on $\Omega$,
the last term in \eqref{tjohej} is non-negative since $\langle n(x), D\varphi(x) \rangle \geq 1$ by Lemma \ref{lemma:varphi}.

Using $\langle n(x), D\varphi(x) \rangle \geq 1$ once more it now only remains to pick $A$ such that $A - \max_{i,j \in \{1, \dots, m\}}|D c_{ij}| \geq \tilde{A}$ in order to fulfill \eqref{tjohej_ny}.
Observe that this choice of $A$ %is independent from $\epsilon$, $B$, $C$ and $\kappa$; it
depends only on the data of the problem.
With $\kappa$ and $A$ now fixed we conclude that for any $\epsilon > 0$ we can choose $B = B(\epsilon, A)$ and then $C = C(A,B,\kappa)$ such that the above calculations hold.

Last, we note that the constructed barrier now satisfies the boundary condition in the classical sense. This suffices as Proposition 7.2 of \cite{CIL92} and the specific form of the function $B_i$ (cf. $(7.4)$ of \cite{CIL92}) then gives that the boundary condition also holds in the viscosity sense.
\end{proof}

\noindent \newline
\begin{proof} [Proof of Proposition \ref{prop:Perron}]
Note first that $w$ is well defined and bounded by the assumption of existence of sub- and supersolutions. Let $w_\ast:=(w_{\ast, 1}, \ldots, w_{\ast,m})$ and $w^\ast:=(w^\ast_1, \ldots, w^\ast_m)$ denote the lower- and upper semi-continuous envelopes of $w=(w_1, \dots, w_m)$, respectively, i.e., the largest LSC function that is dominated by $w$ and the smallest USC function that dominates $w$, respectively. By definition and Theorem \ref{thm:comparison} we have, for any $i, \hat x$ and $\epsilon>0$ fixed,
$$
w_\ast \leq w^\ast, \qquad u^{i,\hat x, \epsilon} \leq w^\ast \quad \mbox{and}\quad w_\ast \leq v^{i,\hat x, \epsilon}.
$$
The essence of Perron's method is to now prove that $w_\ast$ is a supersolution and $w^\ast $ a subsolution to \eqref{problem}, implying that also $w^\ast \leq w_\ast$ by Theorem \ref{thm:comparison} and thus that $w = w_\ast =w^\ast$ is a solution to \eqref{problem}.

To establish the subsolution property, we need to show
\begin{align*}
(i)&\quad \mbox{$\min \left\{a+ {F_i}(t,x,w^\ast_i(t,x), p,X),  w^\ast_{i}\left( x,t\right) - \M_i w^\ast(t,x) \right\} \leq 0,$} \,\mbox{$ (t,x) \in (0,T)\times \Omega, (a,p,X) \in \supjetbar w^\ast_i(t,x)$} \\
(ii)&\quad \mbox{$\min \left\{ a + {F_i}(t,x,w^\ast(t,x), p, X ),  w^\ast_{i}(t,x) - \M_i w^\ast(t,x) \right\}$}\wedge \mbox{$ \langle n(x), p \rangle + f_i(t,x,w_i^\ast(x)) \leq 0 $}, \\
&\quad  \mbox{for $(t,x) \in (0,T)\times \partial \Omega, (a,p,X) \in \supjetbar w^\ast_i(t,x)$ } \\
(iii)& \quad \mbox{$w^\ast_i(0,x) \leq g_i(x)$, $x\in \bar \Omega$}
\end{align*}
Statement $(iii)$ follows immediately from the assumptions and Theorem \ref{thm:comparison}. Indeed, $w_i \leq v^{i,\hat x, \epsilon}_i$ for any $i, \hat x,$ and $\epsilon$ and thus $w^\ast_i(0,\hat x) \leq \inf_{\epsilon} (v_i^{i,\hat x, \epsilon})^\ast(0,\hat x) =g_i(\hat x)$ (where the last equality is by assumption).

Concerning statement $(i)$, we first note that by the definition of $w^\ast_{i}$ there exists a sequence \begin{equation} \label{eq:sequence1}
(t_n, x_n, u_i^n(t_n,x_n)) \to (t,x,w_i^\ast(t,x))
\end{equation}
with each $u^n$ being a subsolution of \eqref{problem}. Assume now that $(a,p,X) \in \supjetbar w^\ast_i(\hat t,\hat x)$ for $(\hat t,\hat x) \in (0,T)\times\Omega$. Then, by the existence of the sequence \eqref{eq:sequence1} and the fact that $w_i^\ast$ is USC we get from Proposition 4.3 of \cite{CIL92} that there exists a sequence
\begin{equation*} %\label{eq:sequence2}
(t_n,x_n, u^n_i, a_n, p_n, X_n) \to (t,x,w^\ast (t,x), a,p,X) \quad  \mbox{with} \quad (a_n,p_n,X_n) \in \supjetbar u^n_i(t_n,x_n).
\end{equation*}
Since $(\hat t,\hat x) \in (0,T)\times\Omega$ we will have $(t_n, x_n) \in (0,T)\times\Omega$ as well for $n$ large enough
and by the subsolution property of $u^n$ and the fact that $F$ is continuous we get
$$
a + F_i(t,x,w_i^\ast, p ,X) = \lim_{n \to \infty} \left( a_n + F_i(t_n,x_n,u^n_i(t_n,x_n),p_n,X_n) \right)\leq 0
$$
and thus $w^\ast$ satisfies $(i)$ above. If $(a,p,X) \in \supjetbar w^\ast_i(\hat t,\hat x)$ for $(\hat t, \hat x) \in (0,T)\times\partial \Omega$, analogous reasoning gives the existence of a sequence such that
$$
\left( a_n + F_i(t_n,x_n,u^n_i(t_n,x_n),p_n,X_n) \right) \wedge B_i(t_n,x_n, u_i^n(t_n,x_n), p_n) \leq 0
$$
for all $n$. Taking the limit as $n \to \infty$, now using also that $f_i$ is continuous, we conclude that also $(ii)$ holds. In particular, we then have $w^\ast=w$.

We now prove that $w_\ast$ is a supersolution following a classical argument by contradiction. More precisely, we show that if $w_\ast$ is \textit{not} a supersolution, then there exists a subsolution strictly greater than $w ^\ast$, contradicting the very definition of $w$.

Starting with the initial condition $(iii)$ we have $w_i (0,\hat x) \geq u_i^{i,\hat x, \epsilon} (0,\hat x)$ for all $\epsilon >0$ and thus $w_{\ast,i}(0,\hat x) \geq  \sup_\epsilon (u_i^{i,\hat x, \epsilon})_\ast (0,\hat x)= \sup_\epsilon u_i^{i,\hat x, \epsilon}(0,\hat x) =g_i(\hat x)$.
Assume now that $w_\ast$ is not a supersolution by violating Definition \ref{def:solution} $(i)$,
%by violating \eqref{def:solution} $(ii)$,
i.e., that for some $(\hat t,\hat x) \in (0,T)\times\Omega$ and $i \in \{1,\dots, m\}$ we have
\begin{align*}
&\min \{a+ F_i(\hat t, \hat x, w_{\ast,i}(\hat t, \hat x), p,X), w_{\ast,i}(\hat t,\hat x) - \M_iw_{\ast,i}(\hat t,\hat x)\}<0
\end{align*}
for some $(a,p,X) \in  \subjetbar w_{\ast,i}(\hat t, \hat x)$.
For $\delta >0$ and such $(a,p,X)$ fixed, construct the function
$$
\tilde w(t,x): =w_{\ast,i}(\hat t, \hat x) + \delta + a (t-\hat t) + \langle p,(x-\hat x)\rangle + \frac{1}{2} \langle X(x-\hat x),(x-\hat x)\rangle- \beta (|t-\hat t| + |x-\hat x|^2).
$$
Since $\M_i w_{\ast} \leq \M_i w^\ast$ it follows from continuity that the function $\tilde w(t,x)$ % := \{\tilde w_1(t,x), \ldots, \tilde w_m(t,x) \}$
is a viscosity subsolution to
\begin{equation} \label{eq:satisfiesabove}
 \min \left\{ \partial _{t} \tilde w(t,x) + {F}_i\left(t,x,\tilde w(t,x), D\tilde w, D^{2}  \tilde w\right),  \tilde w(t,x)_{i}- \M_i w^\ast(t,x) \right\} =0
\end{equation}
for $(t,x) \in Q_R := \{(t,x) : |t-\hat t| + |x-\hat x|^2 <R\}$ and $\delta$, $\beta$ and $R$ sufficiently small.
(If $t \neq \hat t$, $\tilde w$ satisfies \eqref{eq:satisfiesabove} in the classical sense.
If $t=\hat t$, $\supjetbar \tilde w(t,x) = \{(a+\beta \eta, p, X) : \eta \in [-1,1]\}$
and the contribution from $\beta \eta$ is harmless if $\beta$ is small enough.)
By definition of $\subjetbar w_{\ast,i}(\hat t, \hat x)$ we have
\begin{align*}
w_i^\ast (t,x) &\geq w_{\ast,i}(t,x) \geq w_{\ast,i}(\hat t, \hat x) + a (t-\hat t) + \langle p, x-\hat x\rangle + \frac{1}{2} \langle X (x-\hat x), x-\hat x \rangle + o (|t-\hat t| + |x-\hat x|^2)  \\
&= \tilde w(t,x) - \delta + \beta (|t-\hat t| + |x-\hat x|^2) + o(|t-\hat t| + |x-\hat x|^2)
\end{align*}
and thus, if we let $\delta = \frac{\beta R}{4}$ and consider $(t,x) \in Q_R \setminus Q_{\frac{R}{2}}$, we get
\begin{align} \label{eq:nojump}
w_i^\ast (t,x) &\geq  \tilde w(t,x) - \frac{\beta R}{4} + \beta (|t-\hat t| + |x-\hat x|^2) + o(|t-\hat t| + |x-\hat x|^2) \notag \\
& \geq \tilde w(t,x) - \frac{\beta R}{4}  + \frac{\beta R}{2}  + o (R).
\end{align}

Now let $\check u = \{\check u_1, \dots, \check u_m\}$ where
$$
\check u_i =\begin{cases}
 \max\{w^{\ast}_i(t,x), \tilde w_i(t,x)\} & $\mbox{if $(t,x)\in Q_R$}$\\
w^\ast_i(t,x) &\mbox{otherwise}	\end{cases} \qquad \mbox{and} \qquad \mbox{$\check u_j (t,x)= w^\ast_j(t,x)$ if $j\neq i$}.
$$
Note that by \eqref{eq:nojump}, there is no jump in $\check u_i$ at $\partial Q_R$ if $R$ is small enough. Since $\check u_i \geq w^\ast_i$ we have
$$\check u_j - \M_j \check u \leq w^\ast_j - \M_j w^\ast$$
for any $j \neq i$. Recalling that $w^\ast_i$ is a subsolution, the above shows that $\check u$ is a subsolution to \eqref{problem} as well (after decreasing $R$ even further if necessary to ensure $Q_R \in (0,T)\times\Omega$).
We now note that, since $(w^\ast_i)_\ast = w_{\ast,i}$, there exists by definition a sequence $(t_n, x_n, w^\ast_i(t_n, x_n)) \to (\hat t, \hat x, w_{\ast,i}(\hat t, \hat x))$. If we follow this sequence we thus find that $\check u_i (t_n, x_n)  = \tilde w_i (t_n, x_n) > w^\ast_i$ for some point $(t_n,x_n)$ sufficiently close to $(\hat t, \hat x)$ (since $\tilde w_i \to w_{\ast,i} + \delta > w_{\ast,i}$). Thus, we have constructed a subsolution which is strictly greater than $w$,  a contradiction.

What remains is to consider if $w_\ast$ fails to be a supersolution by violating condition $(ii)$, i.e., if there exists $(\hat t, \hat x) \in (0,T)\times\partial \Omega$ and $(a,p,X) \in \subjetbar w_\ast(\hat t, \hat x)$ s.t.
\begin{align}\label{eq:boundarycontraassumption}
&\min \{a+ F_i(\hat t, \hat x, w_{\ast,i}(\hat t, \hat x), p,X), w_{\ast,i}(\hat t,\hat x) - \M_iw_{\ast,i}(\hat t,\hat x)\} \vee B_i(\hat t, \hat x, w_{\ast,i}(\hat t, \hat x),p)<0.
\end{align}
However, if \eqref{eq:boundarycontraassumption} holds, continuity of $f_i$ and the smoothness of $\partial \Omega$ gives that $B_i(t, x, \tilde w(t,x), p) \leq 0$ (in the classical sense and thus in the viscosity sense by Proposition 7.2 of \cite{CIL92}) for $(t,x) \in \partial \Omega$ and sufficiently close to $(\hat t, \hat x)$. We then conclude as above that there exists $\delta, \beta$, and $R$ s.t.
$\tilde w$ satisfies the subsolution property \eqref{eq:satisfiesabove} for $(t,x) \in (0,T)\times\Omega$ sufficiently close to $(\hat t, \hat x)$ and thus that $\check u$ is a subsolution. Again, we have constructed a subsolution which dominates $w^\ast$, contradicting the very definition of $w$. The proof is complete.
\end{proof}

\noindent \newline
{\bf Acknowledgement.}
%We thank two anonymous referees for careful reading of the manuscript and for comments which have significantly improved the presentation.
Marcus Olofsson was financed in full and Niklas Lundstr\"om in part by the Swedish research council (grant no. 2018-03743). We gratefully acknowledge this support.

%%%%%%%%%%%%%%%%%%%%%%%%%%%%%%%%%%%%%%%%%%%%
%%%%%%%%%%%%%%%%%%%%%%%%%%%%%%%%%%%%%%%%%%%%
%%%%%%%%%%%%%%%%%%%%%%%%%%%%%%%%%%%%%%%%%%%%
%%%%%%%%%%%%%%%%%%%%%%%%%%%%%%%%%%%%%%%%%%%%
%%%%%%%%%%%%%%%%%%%%%%%%%%%%%%%%%%%%%%%%%%%%
%%%%%%%%%%%%%%%%%%%%%%%%%%%%%%%%%%%%%%%%%%%%
%%%%%%%%%%%%%%%%%%%%%%%%%%%%%%%%%%%%%%%%%%%%
%%%%%%%%%%%%%%%%%%%%%%%%%%%%%%%%%%%%%%%%%%%%
%%%%%%%%%%%%%%%%%%%%%%%%%%%%%%%%%%%%%%%%%%%%
%%%%%%%%%%%%%%%%%%%%%%%%%%%%%%%%%%%%%%%%%%%%

\end{document}